\documentclass{elsarticle}

\makeatletter
\def\ps@pprintTitle{%
\let\@oddhead\@empty
\let\@evenhead\@empty
\def\@oddfoot{}%
\let\@evenfoot\@oddfoot}
\makeatother

\usepackage{amsmath}
\usepackage{color} 
\usepackage{amssymb}
\usepackage{amsthm}
\usepackage{MnSymbol}
\usepackage{pgf,tikz}
\usepackage{caption}

\usepackage{chngcntr}
\usepackage{lipsum}
\usepackage{color}

\usepackage[utf8]{inputenc}
\usepackage{pstricks-add}

\usetikzlibrary{arrows}
\theoremstyle{plain}
\newtheorem{thm}{Theorem}[section]
\newtheorem{lem}{Lemma}[section]
\newtheorem{cor}{Corollary}[section]
\newtheorem{rem}{Remark}[section]
\newtheorem{prop}{Proposition}[section]

\newtheorem{Fact}{Fact}[section]

\newtheorem{Notation}{Notation}[section]

\newtheorem{Observation}{Observation}[section]
\numberwithin{equation}{section}

\begin{document}
\nocite{*}
\sloppy

\begin{frontmatter}

\title{Tournaments with maximal decomposability}
\author{Cherifa Ben Salha \fnref{}}
\address{University of Carthage, Faculty of Sciences of Bizerte, Bizerte, Tunisia}
\ead{cherifa.bensalha@fsb.u-carthage.tn}

\begin{abstract} 
Given a tournament $T$, a module of $T$ is a subset $M$ of $V(T)$ such that for $x, y\in M$ and $v\in V(T)\setminus M$, $(x,v)\in A(T)$ if and only if $(y,v)\in A(T)$. The trivial modules of $T$ are $\emptyset$, $\{u\}$ $(u\in V(T))$ and $V(T)$. The tournament $T$ is indecomposable if all its modules are trivial; otherwise it is decomposable. The decomposability index of $T$, denoted by $\delta(T)$, is the smallest number of arcs of $T$ that must be reversed to make $T$ indecomposable. In  a previous paper, we proved that for $n \geq 5$, we have  $\delta(n) = \left\lceil \frac{n+1}{4} \right\rceil$, where $\delta(n)$ is the maximum of $\delta(T)$ over the tournaments $T$ with $n$ vertices.   In this paper, we characterize the tournaments $T$ with $\delta$-maximal decomposability, i.e., such that $\delta(T)=\delta(\vert T\vert)$.   
\end{abstract}
\end{frontmatter}

\section{Introduction}
Arc reversal problems in tournaments have been studied by several authors under  different considerations.
The general problem consists of transforming a tournament $T$ into a tournament satisfying a given property (P) by reversing a minimum number of arcs. 
 This number is the distance of $T$ to the set of the tournaments defined on the same vertex set as $T$, and satisfying Property (P). It can be considered as a measure of the negation of Property (P). 
 
 S.J. Kirkland \cite{K} was interested in this problem when Property (P) is reducibility. He established that every $n$-vertex tournament $T$ can be made reducible (i.e. non strongly connected) by reversing at most $\left\lfloor \frac{n-1}{2} \right\rfloor$ arcs, and that this bound is best possible. He also characterized the tournaments $T$ with maximal strong connectivity, i.e., the tournaments $T$ such that $s(T)= \left\lfloor \frac{v(T)-1}{2} \right\rfloor$, where $s(T)$ is the minimum number of arcs that must be reversed to make $T$ reducible. He proved that these tournaments are the regular and almost regular tournaments. 
 
When Property (P) is decomposability, the problem was introduced by V.~M\"{u}ller and J. Pelant \cite{MP}  because of its relation with the strongly homogeneous tournaments
 whose existence is equivalent to that of skew Hadamard matrices (e.g. see \cite{BR}). They introduce the indecomposability index of a tournament $T$ (that they call arrow-simplicity) as the smallest integer $i(T)$ for which there exist $i(T)$ arcs of $T$ whose reversal makes T decomposable. They
prove that $i(T) \leq \frac{1}{2}(v(T)  -1)$, and that the upper bound is uniquely reached by the strongly homogeneous tournaments. This gives a new characteristic property of strongly homogeneous tournaments: the strongly homogeneous tournaments are the tournaments with maximal indecomposability. 

In this paper, we consider the negation of the property considered by V. M\"{u}ller and J. Pelant \cite{MP}, i.e., indecomposability.
Given a tournament $T$ on at least 5 vertices, H. Belkhechine \cite{Index} introduced the decomposability index of $T$, denoted by $\delta(T)$, as the minimum number of arcs that must be reversed to make $T$ indecomposable, and he asked to characterize the tournaments $T$ with $\delta$-maximal decomposbility, i.e., the tournaments $T$ such that $\delta(T) = \delta(v(T))$, where for every integer $n \geq 5$, $\delta(n)$ is the maximum of $\delta(U)$ over the $n$-vertex tournaments $U$. In this paper, we characterize such tournaments (see Theorems~\ref{p4n}--\ref{p4n+3}).

\section{Preliminaries}

A {\it tournament} $T$ consists of a finite set $V(T)$ of {\it vertices} together with a set $A(T)$ of ordered pairs of distinct vertices, called {\it arcs}, such that for all $x \neq y \in V(T)$, $(x,y) \in A(T)$ if and only if $(y,x) \not\in A(T)$. Such a tournament is denoted by $(V(T), A(T))$. The {\it cardinality} of $T$, denoted by $v(T)$, is that of $V(T)$. Given a tournament $T$, with each subset $X$ of $V(T)$ is associated the {\it subtournament} $T[X] = (X, A(T) \cap (X \times X))$ of $T$ induced by $X$. Two tournaments $T$ and $T'$ are {\it isomorphic}, which is denoted by $T \simeq T'$, if there exists an {\it isomorphism} from $T$ onto $T'$, that is, a bijection $f$ from $V(T)$ onto $V(T')$ such that for every $x, y \in V(T)$, $(x, y) \in A(T)$ if and only if $(f(x), f(y)) \in A(T')$. With each tournament $T$ is associated its {\it dual} tournament $T^{\star}$ defined by $V(T^{\star})= V(T)$ and $A(T^{\star}) =\{(x,y) \colon\ (y,x)\in A(T)\}$. 

A {\it transitive} tournament is a tournament $T$ such that for every $x,y,z \in V(T)$, if $(x,y) \in A(T)$ and  $(y,z) \in A(T)$, then $(x,z) \in A(T)$. Let $n$ be a positive integer. We denote by $\underline{n}$ the transitive tournament whose vertex set is $\{0, \ldots, n-1\}$ and whose arcs are the ordered pairs $(i,j)$ such that $0 \leq i < j \leq n-1$. Up to isomorphism, $\underline{n}$ is the unique transitive tournament with $n$ vertices.

The paper is based on the following notions. Given a tournament $T$, a subset $M$ of $V(T)$ is a {\it module} \cite{Spinrad} (or a {\it clan} \cite{E} or an {\it interval} \cite{I}) of $T$ provided that for every $x,y \in M$ and for every $v \in V(T) \setminus M$, $(v,x) \in A(T)$ if and only if $(v,y) \in A(T)$. For example, $\emptyset$, $\{x\}$, where $x \in V(T)$, and $V(T)$ are modules of $T$, called {\it trivial} modules. 
A tournament is {\it indecomposable} \cite{I, ST} (or {\it prime} \cite{Spinrad} or {\it primitive} \cite{E}) if all its modules are trivial; otherwise it is {\it decomposable}. Let us consider the tournaments of small sizes.
The tournaments with at most two vertices are clearly indecomposable. The tournaments $\underline{3}$ and $C_{3} = (\{0,1,2\}, \{(0,1), (1,2), (2,0)\})$ are, up to isomorphism, the unique tournaments with three vertices. The tournament $C_{3}$ is indecomposable, whereas $\underline{3}$ is decomposable.
Up to isomorphism, there are exactly four tournaments with four vertices, all of then are decomposable. These tournaments are described in Fact~\ref{tournois4} (see  also Figure~\ref{t4}).  
Up to isomorphism, the indecomposable tournaments with five vertices are
$U_{5}=(\{0,\ldots,4\},A(\underline{3})\cup \{(3,0),(1,3),(2,3),(4,0),(4,1),(2,4),(3,4)\}$,
$V_{5}=(\{0,\ldots,4\}, (A(U_{5}) \setminus \{(3,4)\})\cup \{(4,3)\})$, and
 $W_{5}=(\{0,\ldots,4\},A(\underline{4})\cup\{(4,0),(4,2),(1,4),(3,4)\})$ (e.g. see \cite{BB}). The twelve tournament with five vertices are described in Fact~\ref{tournois5} (see also Figure~\ref{t5}). 
 Let $T$ and $T'$ be isomorphic tournaments. The tournament $T$ is indecomposable if and only if $T'$ is. 
Similarly, a tournament $T$ and its dual share the same modules. In particular, $T$ is indecomposable if and only if $T^{\star}$ is.

Let $T$ be a tournament. An {\it inversion} of an arc $a = (x,y) \in A(T)$ consists of replacing the arc $a$ by $a^{\star}$ in $A(T)$, where $a^{\star} = (y,x)$. The tournament obtained from $T$ after reversing the arc $a$ is denoted by ${\rm Inv}(T,a)$. Thus ${\rm Inv}(T,a) = (V(T), (A(T) \setminus \{a\}) \cup \{a^{\star}\})$. More generally, for $B \subseteq A(T)$, we denote by ${\rm Inv}(T, B)$ the tournament obtained from $T$ after reversing all the arcs of $B$, that is ${\rm Inv}(T, B)= (V(T), (A(T) \setminus B) \cup B^{\star})$, where $B^{\star} = \{b^{\star} \colon\ b \in B\}$. For example, $T^{\star}={\rm Inv}(T, A(T))$.
Given a tournament $T$ with at least five vertices,
the {\it decomposability index} of $T$, denoted by $\delta(T)$, has been introduced by  H.~Belkhechine \cite{Index} as the smallest integer $m$ for which there exists $B \subseteq A(T)$ such that $|B|=m$ and ${\rm Inv}(T, B)$ is indecomposable. The index $\delta(T)$ is well-defined because  for every integer $n \geq 5$, there exist indecomposable tournaments with $n$ vertices (e.g. see \cite{Index}). Notice that $\delta(T) = \delta(T^{\star})$.  
Similarly, isomorphic tournaments have the same decomposability index. For $n \geq 5$, let $\delta(n)$ be the maximum of $\delta(T)$ over the tournaments $T$ with $n$ vertices. H. Belkhechine and I \cite{Index2} have proved that
\begin{equation} \label{Q} 
\delta(n) = \left\lceil \frac{n+1}{4} \right\rceil \ \text{for every integer} \ n \geq 5 \ (\text{see Theorem}~\ref{deltan}). 
\end{equation}
In fact, (\ref{Q}) is \cite[Conjecture 6.2]{Index}.
Given a tournament $T$ with at least five vertices, let us say that $T$ has {\it $\delta$-maximal decomposability} when $\delta(T) = \delta(v(T))$, i.e., when $\delta(T) = \left\lceil \frac{v(T)+1}{4} \right\rceil$ (see (\ref{Q})). For example, the transitive tournaments with at least five vertices have $\delta$-maximal decomposability because $\delta(\underline{n}) = \delta(n)$ for every integer $n \geq 5$ (see \cite{Index}). But transitive tournaments are not the only tournaments with $\delta$-maximal decomposability. The aim of this paper is to characterize the tournaments with such maximality. This answers a problem posed by H. Belkhechine \cite{Index}. The notions of co-module and co-modular decomposition form important tools for this purpose.
These notions have been introduced in \cite{Index2} as follows.
Given a tournament $T$, a {\it co-module} of $T$ is a subset $M$ of $V(T)$ such that $M$ or $V(T)\setminus M$ is a nontrivial module of $T$.  A {\it co-modular decomposition} of $T$ is a set of pairwise disjoint co-modules of $T$.
Observe that a tournament $T$ is decomposable if and only if it admits a nonempty co-modular decomposition.
A {\it $\Delta$-decomposition} of $T$ is a co-modular decomposition $D$ of $T$ which is of maximum size.
 The {\it co-modular index} of a tournament $T$, denoted by $\Delta(T)$, is the size of a $\Delta$-decomposition of $T$.
  For instance, a tournament T is decomposable if and only if $ \Delta(T) \geq 2$.
  For every isomorphic tournaments $T$ and $T'$, we have $\Delta(T) = \Delta (T') = \Delta(T^{\star})$. 
 For $n \geq 3$, let $\Delta(n)$ be the maximum of $\Delta(T)$ over the tournaments $T$ with $n$ vertices. In \cite{Index2}, we have proved that 
 \begin{equation} \label{Q'} 
 \Delta(n) = \left\lceil \frac{n+1}{2} \right\rceil \ \text{for every integer} \ n \geq 3 \ (\text{see Theorem}~\ref{deltan}), 
 \end{equation}   
 and that the decomposability index and the co-modular one are directly related as follows:
 \begin{equation} \label{Q''} 
 \delta(T) = \left\lceil \frac{\Delta(T)}{2} \right\rceil \ (\text{see Theorem}~\ref{deltan}) 
 \end{equation}
 for every tournament $T$ with at least five vertices. This relation is certainly our most important result in \cite{Index2}, although our original purpose was to prove or disprove \cite[Conjecture 6.2]{Index}, i.e. (\ref{Q}). Indeed, (\ref{Q}) is a direct consequence of (\ref{Q'}) and (\ref{Q''}).
 
 It follows from (\ref{Q''}) that the tournaments $T$ with at least five vertices such that $\Delta(T) =\Delta(v(T))$ have $\delta$-maximal decomposability. This leads us to consider the tournaments having {\it $\Delta$-maximal decomposability}, i.e., the tournaments $T$ with at least three vertices such that $\Delta(T) =\Delta(v(T))$. The tournaments with $\Delta$-maximal or $\delta$-maximal decomposability will be described in terms of lexicographic sums of tournaments of small sizes. This notion is introduced in Section~\ref{lexi}. In Section~\ref{urmc}, we review some required results on co-modules and co-modular decompositions. We characterize  the tournament with $\Delta$-maximal decomposability in Section~\ref{Deltamax}, and those with $\delta$-maximal decomposability in Section~\ref{deltamax}.
 
 \section{Co-modules and co-modular decompositions}\label{urmc}

\begin{Notation} \normalfont 
	Given a co-modular decomposition $D$ of a tournament $T$, we denote by ${\rm sg}(D)$ the set of the singletons of $D$. Moreover, $D\setminus {\rm sg}(D)$ is denoted by $D_{\geq 2}$. 	
\end{Notation}
\begin{lem} [\cite{Index2}] \label{comod part}
 Given a decomposable tournament $T$, consider a co-modular decomposition $D$ of $T$. The following assertions are satisfied.
 \begin{enumerate}
 \item The tournament $T$ admits at most two singletons which are co-modules of $T$. In particular $|{\rm sg}(D)| \leq 2$. 
  \item If $D$ contains an element $M$ which is not a module of $T$, then the elements of $D \setminus \{M\}$ are nontrivial modules of $T$.
 \end{enumerate}
\end{lem}
The following corollary is a direct consequence of Lemma \ref{comod part}.

\begin{cor}\label{rcomod}
 Given a decomposable tournament $T$, consider a co-modular decomposition $D$ of $T$.
 We have $\vert {\rm sg}(D)\vert \leq 2$. Moreover, if ${\rm sg}(D) \neq \varnothing$, then the elements of $D_{\geq 2}$ are nontrivial modules of $T$.
\end{cor}

Let $T$ be a tournament. A {\it minimal co-module} of $T$ is a co-module $M$ of $T$ which is minimal in the set of co-modules of $T$ ordered by inclusion, i.e., such that $M$ does not contain any other co-module of $T$. Similarly, a {\it minimal nontrivial module} of a tournament $T$ is a nontrivial module of $T$ which is minimal in the set of nontrivial modules of $T$ ordered by inclusion.

\begin{Notation} \normalfont
  Given a tournament $T$, the set of minimal co-modules of $T$ is denoted by $\text{mc}(T)$.
\end{Notation}
A {\it minimal $\Delta$-decomposition} (or a  {\it $\delta$-decomposition}) \cite{Index2} of a tournament $T$ is a $\Delta$-decomposition $D$ of $T$ in which every element is a minimal co-module of $T$, i.e., for every $M\in D$, $M\in \text{mc}(T)$.
For example, for every integer $n \geq 3$, we have $\text{mc}(\underline{n}) = \{\{0\}, \{n-1\}\} \cup \{\{i,i+1\}: 1 \leq i \leq n-3\}$. Moreover, $\{\{0\}, \{n-1\}\} \cup \{\{i,i+1\}: 1 \leq i \leq n-3 \ \text{and} \ i \ \text{is odd}\}$ is a $\delta$-decomposition of $\underline{n}$.

\begin{rem}[\cite{Index2}] \label{fact1 min mod comod} \normalfont
Given a nontrivial module $M$ of a tournament $T$, the following assertions are satisfied.
\begin{enumerate}
\item If $M$ is a minimal co-module of $T$, then $M$ is a minimal nontrivial module of $T$. \item If $M$ is a minimal nontrivial module of a tournament $T$, then $T[M]$ is indecomposable.
\end{enumerate}
\end{rem}

\begin{Observation}  \label{rr2}
	Let $T$ be a tournament and let $M \in {\rm mc}(T)$. The following assertions are satisfied.
	\begin{enumerate}
		\item If $\vert M\vert =3$, then $T[M]\simeq C_{3}$.
		\item If $M$ is a nontrivial module of $T$, then $\vert M\vert \neq 4$.
		\item If $M$ is a nontrivial module of $T$ and $\vert M \vert =5$, then $T[M]\simeq U_{5}$, $V_{5}$ or $W_{5}$.
	\end{enumerate}
\end{Observation}
\begin{proof}
	For the first assertion, suppose $|M|=3$. If $M$ is a module of $T$, then by Remark~\ref{fact1 min mod comod}, $T[M]\simeq C_{3}$ because $C_3$ is, up to isomorphism, the unique indecomposable $3$-vertex tournament. Now, suppose that $M$ is not a module of $T$. Suppose for a contradiction that $T[M] \not\simeq C_3$. Up to isomorphism, we may assume $T[M]= \underline{3}$. Recall that $\overline{M}$ is a module of $T$ because $M$ is not. By examining the different possible adjacency relations between $\overline{M}$ and the vertices of $M$, it is easily seen that $\{0,1\}$, $\{1,2\}$, $V(T) \setminus \{0\}$, or $V(T) \setminus \{2\}$ is a nontrivial module of $T$, which contradicts the minimality of the co-module $M$ of $T$. 
	
	The second assertion follows from Remark~\ref{fact1 min mod comod} and from the fact that de $4$-vertex tournaments are decomposable. The third one follows from Remark~\ref{fact1 min mod comod} and from the fact that the indecomposable tournaments with five vertices are, up to isomophism, $U_5$, $V_5$ and $W_5$.  
	\end{proof}

The following theorem summarizes the main results of \cite{Index2}.

\begin{thm} [\cite{Index2}] \label{deltan} 
	For every tournament $T$ with $v(T)\geq 5$, we have $\delta(T) = \left\lceil \frac{\Delta(T)}{2} \right\rceil$. Moreover, for every integer $n \geq 3$, we have $\Delta(n) = \left\lceil \frac{n+1}{2} \right\rceil$. In particular, for every integer $n \geq 5$, we have $\delta(n) = \delta(\underline{n})= \left\lceil \frac{\Delta(n)}{2} \right\rceil = \left\lceil \frac{n+1}{4} \right\rceil$.
\end{thm}

\begin{cor} \label{cor}
	Let $T$ be a decomposable tournament with $n$ vertices, where $n\geq 5$. Since $\delta(T) = \left\lceil \frac{\Delta(T)}{2} \right\rceil$, $\delta(n) = \left\lceil \frac{n+1}{4} \right\rceil$ and $\Delta(n)=\left\lceil \frac{n+1}{2} \right\rceil$ (see Theorem~\ref{deltan}), the following assertions are satisfied.
	\begin{enumerate}
		\item If $n\equiv 0 \text{ or } 1 \mod 4$, then $\delta(T)=\delta(n)$ if and only if $\Delta(T)=\Delta(n)$.
		\item If $n\equiv 2\text{ or } 3 \mod 4$, then $\delta(T)=\delta(n)$ if and only if $\Delta(T)=\Delta(n)$ or $\Delta(T)=\Delta(n)-1$.
	\end{enumerate}
\end{cor}

\section{Lexicographic sums of tournaments}\label{lexi}
Given a tournament $T$, consider a family  $(T_{i})_{i\in V(T)}$  of tournaments whose vertex sets $V_{i}= V (T_{i})$ are pairwise disjoint. The {\it lexicographic sum}   of the tournaments $T_{i}$ over the tournament $T$, denoted by $\Sigma _{i\in T} T_{i}$,  is the tournament defined on $V(\Sigma _{i\in T} T_{i})=\displaystyle\bigcup _{i\in V(T)}  V_{i}$ as follows.
Given $x\in V(T_{i})$ and $y\in V(T_{j})$, where $i,j\in V(T)$,
\begin{center}
	$(x,y)\in A(\Sigma _{i\in T} T_{i}) $ if $\begin{cases} i=j \text{ and } (x,y)\in A(T_{i}), \\
	i\neq j \text{ and } (i,j)\in A(T). \end{cases}$
\end{center}
Observe that for every $i \in V(T)$, $V(T_i)$ is a module of the lexicographic sum $\Sigma _{i\in T} T_{i}$.
When $V(T)=\{0,\ldots,n-1\}$, $\Sigma _{i\in T} T_{i}$ is also noted by $T(T_{0},...,T_{n-1})$. \index{$T(T_{0},...,T_{n-1})$} 
If the vertex sets $V (T_{i})$ are
not pairwise disjoint, then we can define the lexicographic sum $\Sigma _{i\in T} T_{i}$ up to
isomorphism by considering a family $(T'_{i})_{i\in V(T)}$ of tournaments defined on pairwise
disjoint vertex sets, and such that $T'_{i}\simeq T_{i}$ for every $i\in V(T)$. Thus, up to
isomorphism, $\Sigma _{i\in T} T_{i}$ is $\Sigma _{i\in T} T'_{i}$.  When the tournaments $T_{i}$ are isomorphic to the same tournament, say $H$, then  $\Sigma _{i\in T} T_{i}$ is called the {\it lexicographic product}  of $H$ by $T$; it is denoted by $T.H$.

We will see that the minimal nontrivial modules of the tournaments with $\Delta$-maximal (resp. $\delta$-maximal) decomposability have cardinalities $2$ or $3$ (resp. $2$, $3$, or $5$), and that most of them have cardinality $2$. These tournaments are then described in terms of lexicographic sums of tournaments which are principally $2$-tournaments. The vertex sets of these $2$-tournaments are then modules of cardinality $2$. A module of cardinality $2$ is called a {\it twin}.

The following facts describe the tournaments with four or five vertices by using lexicographic sums.

\begin{Fact} \label{tournois4}
	Up to isomorphism, the $4$-vertex tournaments are $\underline{4}$, $\underline{2}(\underline{1}, C_3)$, $\underline{2}(C_3, \underline{1})$ and $C_3(\underline{1}, \underline{1}, \underline{2})$ (see Figure~\ref{t4}). 

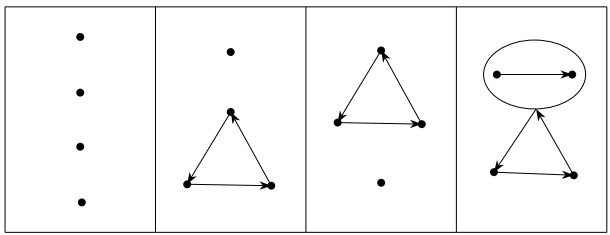
\begin{figure}[h]
	\begin{center}
		\psset{xunit=1cm,yunit=1cm,algebraic=true,dotstyle=o,dotsize=3pt 0,linewidth=0.1pt,arrowsize=3pt 2,arrowinset=0.3}
		\begin{pspicture*}(-1.12,-1.16)(7.12,2.16)
		\pspolygon[linecolor=white,fillcolor=white,fillstyle=solid,opacity=0.1](1.42,-0.36)(2.54,-0.38)(2,0.6)
		\pspolygon[linecolor=white,fillcolor=white,fillstyle=solid,opacity=0.1](3.42,0.46)(4.54,0.44)(4,1.42)
		\psline(-1,-1)(-1,2)
		\psline[linecolor=white](1.42,-0.36)(2.54,-0.38)
		\psline[linecolor=white](2.54,-0.38)(2,0.6)
		\psline[linecolor=white](2,0.6)(1.42,-0.36)
		\psline{->}(1.42,-0.36)(2.54,-0.38)
		\psline{->}(2.54,-0.38)(2,0.6)
		\psline{->}(2,0.6)(1.42,-0.36)
		\psline[linecolor=white](3.42,0.46)(4.54,0.44)
		\psline[linecolor=white](4.54,0.44)(4,1.42)
		\psline[linecolor=white](4,1.42)(3.42,0.46)
		\psline{->}(3.42,0.46)(4.54,0.44)
		\psline{->}(4.54,0.44)(4,1.42)
		\psline{->}(4,1.42)(3.42,0.46)
		\rput{0}(6.04,1.1){\psellipse(0,0)(0.68,0.46)}
		\psline{->}(5.54,1.1)(6.54,1.1)
		\psline{->}(5.5,-0.2)(6.56,-0.24)
		\psline{->}(6.06,0.64)(5.5,-0.2)
		\psline{->}(6.56,-0.24)(6.06,0.64)
		\psline(7,-1)(7,2)
		\psline(7,2)(-1,2)
		\psline(-1,-1)(7,-1)
		\psline(1,2)(1,-1)
		\psline(3,-1)(3,2)
		\psline(5,2)(5,-1)
		\begin{scriptsize}
		\psdots[dotstyle=*](0,1.6)
		\psdots[dotstyle=*](0,0.86)
		\psdots[dotstyle=*](0,0.14)
		\psdots[dotstyle=*](0.02,-0.6)
		\psdots[dotstyle=*](1.42,-0.36)
		\psdots[dotstyle=*](2.54,-0.38)
		\psdots[dotstyle=*](2,0.6)
		\psdots[dotstyle=*](2,1.4)
		\psdots[dotstyle=*](4,-0.34)
		\psdots[dotstyle=*](3.42,0.46)
		\psdots[dotstyle=*](4.54,0.44)
		\psdots[dotstyle=*](4,1.42)
		\psdots[dotstyle=*](5.54,1.1)
		\psdots[dotstyle=*](6.54,1.1)
		\psdots[dotstyle=*](5.5,-0.2)
		\psdots[dotstyle=*](6.56,-0.24)
		\end{scriptsize}
		\end{pspicture*}
		\caption{The 4-vertex tournaments (missing arcs are oriented from higher to lower).}\label{t4} 
	\end{center}
\end{figure}
\end{Fact}

\begin{Fact} \label{tournois5}
	Up to isomorphism, the $5$-vertex tournaments are $\underline{5}$, $\underline{3}(\underline{1}, \underline{1}, C_3)$, $\underline{3}(C_3,\underline{1}, \underline{1})$, $\underline{3}(\underline{1}, C_3,\underline{1})$, $\underline{2}(\underline{1}, C_4)$, $\underline{2}(C_4, \underline{1})$, $C_3(\underline{1}, \underline{2}, \underline{2})$, $C_3(\underline{1}, \underline{1}, C_3)$, $C_3(\underline{1}, \underline{1}, \underline{3})$, $U_5$, $V_5$, and $W_5$ (see Figure~\ref{t5}).
	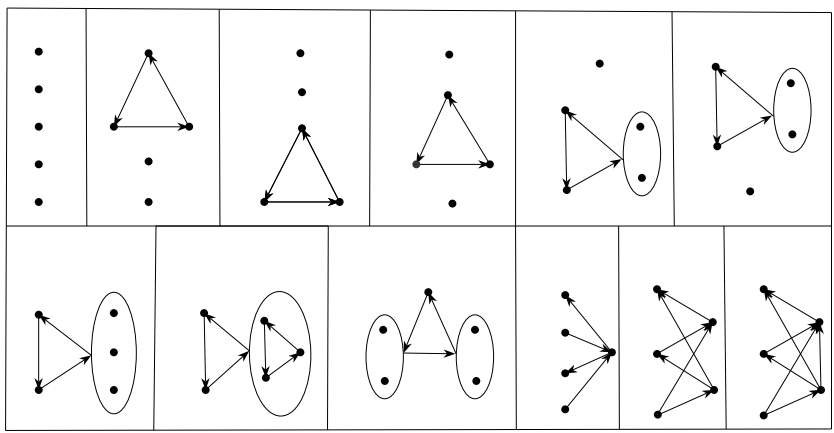
\begin{figure}[!h]
		\begin{center}
		\newrgbcolor{tttttt}{0.2 0.2 0.2}
		\psset{xunit=1.0cm,yunit=1.0cm,algebraic=true,dotstyle=o,dotsize=3pt 0,linewidth=0.1pt,arrowsize=3pt 2,arrowinset=0.3}
		\begin{pspicture*}(-3.54,-0.22)(7.68,5.78)
		\pspolygon[linecolor=white,fillcolor=white,fillstyle=solid,opacity=0.1](-1,4)(-2,4)(-1.54,4.98)
		\psline[linecolor=white](-1,4)(-2,4)
		\psline[linecolor=white](-2,4)(-1.54,4.98)
		\psline[linecolor=white](-1.54,4.98)(-1,4)
		\psline{->}(1,3)(0.5,3.98)
		\psline{->}(0.5,3.98)(0,3)
		\psline{->}(0,3)(1,3)
		\psline{->}(-1,4)(-1.54,4.98)
		\psline{->}(-1.54,4.98)(-2,4)
		\psline{->}(-2,4)(-1,4)
		\psline{->}(3,3.5)(2.44,4.42)
		\psline{->}(2.44,4.42)(2.02,3.5)
		\psline{->}(2.02,3.5)(3,3.5)
		\rput{90}(5.02,3.64){\psellipse(0,0)(0.56,0.25)}
		\psline{->}(4.77,3.56)(4,4.22)
		\psline{->}(4,4.22)(4.02,3.16)
		\psline{->}(4.02,3.16)(4.75,3.57)
		\rput{90}(7.02,4.22){\psellipse(0,0)(0.56,0.25)}
		\psline{->}(6.77,4.14)(6,4.8)
		\psline{->}(6,4.8)(6.02,3.74)
		\psline{->}(6.02,3.74)(6.75,4.15)
		\rput{90}(-2,0.99){\psellipse(0,0)(0.81,0.3)}
		\psline{->}(-2.3,0.96)(-3,1.5)
		\psline{->}(-3,1.5)(-3,0.5)
		\psline{->}(-3,0.5)(-2.3,0.96)
		\psline{->}(1,3)(0.5,3.98)
		\psline{->}(0.5,3.98)(0,3)
		\psline{->}(0,3)(1,3)
		\psline{->}(-0.8,1.52)(-0.78,0.5)
		\psline{->}(0.48,1)(0,1.42)
		\rput{-89.2}(0.21,0.98){\psellipse(0,0)(0.83,0.41)}
		\psline{->}(0,1.42)(0.02,0.66)
		\psline{->}(0.02,0.66)(0.48,1)
		\psline{->}(-0.2,1.02)(-0.8,1.52)
		\psline{->}(-0.78,0.5)(-0.2,1.02)
		\rput{90}(2.8,0.94){\psellipse(0,0)(0.56,0.25)}
		\rput{90}(1.6,0.94){\psellipse(0,0)(0.56,0.25)}
		\psline{->}(2.18,1.8)(1.85,0.99)
		\psline{->}(1.85,0.99)(2.54,0.98)
		\psline{->}(2.55,0.98)(2.18,1.8)
		\psline{->}(4.62,1)(4,1.76)
		\psline{->}(4,1.26)(4.62,1)
		\psline{->}(4.62,1)(4,0.72)
		\psline{->}(4,0.24)(4.62,1)
		\psline{->}(5.22,0.16)(5.98,0.5)
		\psline{->}(5.98,0.5)(5.22,0.98)
		\psline{->}(5.22,0.98)(5.96,1.4)
		\psline{->}(5.96,1.4)(5.22,1.84)
		\psline{->}(5.22,0.16)(5.96,1.4)
		\psline{->}(5.98,0.5)(5.22,1.84)
		\psline{->}(6.64,0.16)(7.4,0.5)
		\psline{->}(7.4,0.5)(6.64,0.98)
		\psline{->}(6.64,0.98)(7.38,1.4)
		\psline{->}(7.38,1.4)(6.64,1.84)
		\psline{->}(6.64,0.16)(7.38,1.4)
		\psline{->}(7.4,0.5)(6.64,1.84)
		\psline{->}(7.4,0.5)(7.38,1.4)
		\psline(-3.44,-0.04)(-3.42,5.58)
		\psline(7.54,-0.02)(-3.44,-0.04)
		\psline(7.54,5.52)(7.54,-0.02)
		\psline(7.54,5.52)(-3.42,5.58)
		\psline(-3.43,2.68)(7.54,2.68)
		\psline(-2.36,2.68)(-2.37,5.57)
		\psline(-0.6,5.56)(-0.59,2.68)
		\psline(1.4,2.68)(1.41,5.55)
		\psline(3.34,5.54)(3.34,2.68)
		\psline(5.42,5.53)(5.45,2.68)
		\psline(-1.44,2.68)(-1.47,-0.04)
		\psline(-1.44,2.68)(0.85,2.68)
		\psline(0.85,2.68)(0.84,-0.03)
		\psline(3.34,2.68)(3.35,-0.03)
		\psline(4.71,2.68)(4.72,-0.03)
		\psline(6.11,2.68)(6.14,-0.02)
		\begin{scriptsize}
		\psdots[dotstyle=*](-3,5)
		\psdots[dotstyle=*](-3,4.5)
		\psdots[dotstyle=*](-3,4)
		\psdots[dotstyle=*](-3,3.5)
		\psdots[dotstyle=*](-3,3)
		\psdots[dotstyle=*](-1.54,4.98)
		\psdots[dotstyle=*](-2,4)
		\psdots[dotstyle=*](-1.54,3.54)
		\psdots[dotstyle=*](-1.54,3)
		\psdots[dotstyle=*](-1,4)
		\psdots[dotstyle=*](1,3)
		\psdots[dotstyle=*](0.5,3.98)
		\psdots[dotstyle=*](0,3)
		\psdots[dotstyle=*](0.48,4.98)
		\psdots[dotstyle=*](0.5,4.46)
		\psdots[dotstyle=*](3,3.5)
		\psdots[dotstyle=*](2.44,4.42)
		\psdots[dotstyle=*,linecolor=tttttt](2.02,3.5)
		\psdots[dotstyle=*](2.46,4.96)
		\psdots[dotstyle=*](2.5,2.98)
		\psdots[dotstyle=*](4.46,4.84)
		\psdots[dotstyle=*](5,4)
		\psdots[dotstyle=*](4,4.22)
		\psdots[dotstyle=*](4.02,3.16)
		\psdots[dotstyle=*](5.02,3.32)
		\psdots[dotstyle=*](6.46,3.14)
		\psdots[dotstyle=*](7,4.58)
		\psdots[dotstyle=*](6,4.8)
		\psdots[dotstyle=*](6.02,3.74)
		\psdots[dotstyle=*](7.02,3.9)
		\psdots[dotstyle=*](-2,1)
		\psdots[dotstyle=*](-2,1.52)
		\psdots[dotstyle=*](-2,0.5)
		\psdots[dotstyle=*](-3,1.5)
		\psdots[dotstyle=*](-3,0.5)
		\psdots[dotstyle=*](-0.8,1.52)
		\psdots[dotstyle=*](-0.78,0.5)
		\psdots[dotstyle=*](0,1.42)
		\psdots[dotstyle=*](0.48,1)
		\psdots[dotstyle=*](0.02,0.66)
		\psdots[dotstyle=*](2.8,1.3)
		\psdots[dotstyle=*](2.82,0.62)
		\psdots[dotstyle=*](1.58,1.3)
		\psdots[dotstyle=*](1.6,0.62)
		\psdots[dotstyle=*](2.18,1.8)
		\psdots[dotstyle=*](4,1.76)
		\psdots[dotstyle=*](4,0.24)
		\psdots[dotstyle=*](4,0.72)
		\psdots[dotstyle=*](4,1.26)
		\psdots[dotstyle=*](4.62,1)
		\psdots[dotstyle=*](5.22,1.84)
		\psdots[dotstyle=*](5.22,0.98)
		\psdots[dotstyle=*](5.96,1.4)
		\psdots[dotstyle=*](5.98,0.5)
		\psdots[dotstyle=*](6.64,1.84)
		\psdots[dotstyle=*](6.64,0.98)
		\psdots[dotstyle=*](6.64,0.16)
		\psdots[dotstyle=*](7.38,1.4)
		\psdots[dotstyle=*](7.4,0.5)
		\psdots[dotstyle=*](5.23,0.17)
		\end{scriptsize}
		\end{pspicture*}
	\caption{The 5-vertex tournaments (missing arcs are oriented from higher to lower).}\label{t5}
	\end{center}
\end{figure}

\end{Fact}

\section{Description of the tournaments with $\Delta$-maximal decomposability}\label{Deltamax}
Let us denote by $\mathcal{T}_{\Delta}$ the class of the tournaments with $\Delta$-maximal decomposability. For every integer $m \geq 3$, we denote by $\mathcal{T}_{\Delta }^m$ the class of the $m$-vertex tournaments of the class $\mathcal{T}_{\Delta}$. Thus $\mathcal{T}_{\Delta} = \displaystyle\bigcup_{m \geq 3} \mathcal{T}_{\Delta}^m$. The description of the class $\mathcal{T}_{\Delta}^m$ depends on the parity of the integer $m$. The class $\mathcal{T}_{\Delta}^m$ is characterized in Proposition~\ref{thmDelta0} when $m$ is even, in Proposition~\ref{thmDelta1} when $m$ is odd. The class $\mathcal{T}_{\Delta}$ is then characterized by Propositions~\ref{thmDelta0} and \ref{thmDelta1}. 

\begin{prop}\label{thmDelta0}
	Up to isomorphism, the tournaments of the class $\mathcal{T}_{\Delta}^{2n}$, where $n \geq 2$, are the tournaments $$\underline{3}(\underline{1},T_{n-1}.\underline{2},\underline{1}),$$ where $T_{n-1}$ is  an arbitrary $(n-1)$-vertex tournament.
\end{prop}
\begin{proof}
	Let $T_{n-1}$ be a tournament with $n-1$ vertices.
	Clearly, $v(\underline{3}(\underline{1},T_{n-1}.\underline{2},\underline{1})) = 2n$ and $\underline{3}(\underline{1},T_{n-1}.\underline{2},\underline{1})$ admits a $\Delta$-decomposition of cardinality $n+1$. Since $\Delta(2n) = n+1$ by Theorem~\ref{deltan}, it follows that
	$\Delta(\underline{3}(\underline{1},T_{n-1}.\underline{2},\underline{1}))=\Delta(2n)$.
	 Thus $\underline{3}(\underline{1},T_{n-1}.\underline{2},\underline{1}) \in \mathcal{T}_{\Delta}^{2n}$.
	
	Conversely, let $T$ be a tournament  of $\mathcal{T}_{\Delta}^{2n}$. Let $D$ be a $\delta$-decomposition of $T$. Since $\Delta(2n) = n+1$ (see Theorem~\ref{deltan}) and $v(T)=2n$, we have $\vert D\vert= n+1$. Since $v(T) = 2n$ and $|D|= n+1$, it follows from Corollary~\ref{rcomod} that $|\text{sg}(D)|=2$ and that the elements of $D_{\geq 2}$ are nontrivial modules of $T$. Moreover, the elements of $D_{\geq 2}$ are twins of $T$ and $\cup D = V(T)$. Therefore, there exists a tournament $T_{n-1}$ with $n-1$ vertices such that $T$ is $\underline{3}(\underline{1},T_{n-1}.\underline{2},\underline{1})$.
\end{proof}
 
 \begin{prop}\label{thmDelta1}
 	Up to isomorphism, the tournaments of the  class $\mathcal{T}_{\Delta}^{2n+1}$, where $n \geq 1$, are the tournaments of one of the following types, where $T_{k}$ is  an arbitrary $k$-vertex tournament:
 	\begin{enumerate}
 		\item[${\rm{T}}1$.] $\underline{2}(\underline{1},T_{n}.\underline{2})$ or its dual;
 		\item[${\rm{T}}2$.] $\underline{3}(\underline{1},T_{n}(\underline{1},\underline{2},\ldots,\underline{2}),\underline{1})$;
 		\item[${\rm{T}}3$.]$\underline{3}(\underline{1},T_{n-1}(C_{3},\underline{2},\ldots,\underline{2}),\underline{1})$ when $n \geq 2$.
 	\end{enumerate}
 \end{prop}
 \begin{proof}
 	Let $T$ be a tournament of type $\text{T}1$, $\text{T}2$, or $\text{T}3$. Clearly, $v(T) = 2n+1$ and $T$ admits a $\Delta$-decomposition of cardinality $n+1$. Since $\Delta(2n+1) = n+1$ (see Theorem~\ref{deltan}) and $v(T) =2n+1$, it follows that $\Delta(T)= \Delta(2n+1)$. Thus $T \in \mathcal{T}_{\Delta}^{2n+1}$.
 	
 	Conversely, let $T$ be a tournament  of $ \mathcal{T}_{\Delta}^{2n+1}$. 
 	Let $D$ be a $\delta$-decomposition of $T$. 
 	Since $|D| = \Delta(T) = \Delta(2n+1)$ and $\Delta(2n+1) =n+1$ (see Theorem~\ref{deltan}), we have $\vert D\vert= n+1$. 
 	Since 
 	\begin{equation} \label{eq0}
 	v(T)=2n+1 \ \ \text{and} \ \ \vert D\vert =n+1,
 	\end{equation}
 	it follows from Corollary~\ref{rcomod} that $\vert \text{sg}(D) \vert \geq 1 $, and that 
 	\begin{equation} \label{eq1}
 	\text{the elements of} \ D_{\geq 2} \ \text{are nontrivial modules of} \ T. 
 	\end{equation}
 	If $\vert \text{sg}(D)\vert =1$, then by (\ref{eq0}) and (\ref{eq1}), $\cup D = V(T)$ and the elements of $D_{\geq 2}$ are twins of $T$. In this instance, $T$ has type $\text{T}1$.
 	Now suppose $\vert \text{sg}(D)\vert =2$. If the elements of $D_{\geq 2}$ are twins of $T$, then $|\cup D| = v(T)-1$ and $T$ has type $\text{T}2$. Now, suppose there exists an element $L \in D_{\geq 2}$ such that $L$ is not a twin  of $T$. In this instance $n \geq 2$, i.e. $v(T) \geq 5$. By (\ref{eq0}) and (\ref{eq1}), $\cup D = V(T)$, $\vert L\vert =3$ and the elements of $D_{\geq 2}\setminus\{ L\}$ are twins of $T$. Moreover, since $L \in D_{\geq 2}$, it follows from (\ref{eq1}) and Observation~\ref{rr2} that $L$ is a module of $T$ such that 
 	$T[L]\simeq C_{3}$. Therefore, the tournament $T$ has type~$\text{T}3$ with $n \geq 2$. 	
 \end{proof}

\section{Description of the tournaments with $\delta$-maximal decomposability}\label{deltamax}
Let us denote by $\mathcal{T}_{\delta}$ the class of the tournaments with $\delta$-maximal decomposability. For every integer $m \geq 5$, we denote by $\mathcal{T}_{\delta}^m$ the class of the $m$-vertex tournaments of the class $\mathcal{T}_{\delta}$. Thus $\mathcal{T}_{\delta} = \displaystyle\bigcup_{m \geq 5} \mathcal{T}_{\delta}^m$. According to Corollary~\ref{cor}, the description of the class $\mathcal{T}_{\delta}^m$ should depend on the congruence of the integer $m$ modulo $4$. The class $\mathcal{T}_{\delta}^m$ is characterized in Theorem~\ref{p4n} when $m \equiv 0 \ \text{mod} \ 4$, in Theorem~\ref{p4n+1} when $m \equiv 1 \ \text{mod} \ 4$, in Theorem~\ref{p4n+2} when $m \equiv 2 \ \text{mod} \ 4$, and in Theorem~\ref{p4n+3} when $m \equiv 3 \ \text{mod} \ 4$. The class $\mathcal{T}_{\delta}$ is then characterized by Theorems~\ref{p4n}--\ref{p4n+3}. 

It follows from the first assertion of Corollary~\ref{cor} that $\mathcal{T}_{\delta}^{4n}= \mathcal{T}_{\Delta}^{4n}$ for every $n \geq 2$, and $\mathcal{T}_{\delta}^{4n+1}= \mathcal{T}_{\Delta}^{4n+1}$ for every $n \geq 1$. Since the classes $\mathcal{T}_{\Delta}^{4n}$ and $\mathcal{T}_{\Delta}^{4n+1}$ are characterized in Propositions~\ref{thmDelta0} and \ref{thmDelta1} respectively, we immediately obtain the following two first theorems. 
 
\begin{thm}\label{p4n}
Up to isomorphism, the tournaments of the  class $\mathcal{T}_{\delta}^{4n}$, where $n \geq 2$, are the tournaments $$\underline{3}(\underline{1},T_{2n-1}.\underline{2},\underline{1}),$$ where $T_{2n-1}$ is  an arbitrary $(2n-1)$-vertex tournament.
\end{thm}

  \begin{thm}\label{p4n+1}
 Up to isomorphism, the tournaments of the  class $\mathcal{T}_{\delta}^{4n+1}$, where $n \geq 1$, are the tournaments with one  of the following forms, where $T_{k}$ is  an arbitrary $k$-vertex tournament:
 \begin{enumerate}
 \item[${\rm{F}}1$.] $\underline{2}(\underline{1},T_{2n}.\underline{2})$ or its dual;
 \item[${\rm{F}}2$.] $\underline{3}(\underline{1},T_{2n}(\underline{1},\underline{2},\ldots,\underline{2}),\underline{1})$;
 \item[${\rm{F}}3$.]$\underline{3}(\underline{1},T_{2n-1}(C_{3},\underline{2},\ldots,\underline{2}),\underline{1})$.
 \end{enumerate}
 \end{thm}

\begin{thm}\label{p4n+2}
 Up to isomorphism,  the tournaments of the class $\mathcal{T}_{\delta}^{4n+2}$, where $n\geq 1$, are the tournaments with one  of the following forms, where $T_{k}$ is  an arbitrary $k$-vertex tournament:
 \begin{enumerate}
 	\item[{\rm F4}.] $T_{2n+1}.\underline{2}$;
 \item[ {\rm F5}.] $C_{3}(\underline{1},\underline{1},T_{2n}.\underline{2})$;
 \item[{\rm F6}.] $\underline{2}(\underline{1}, T_{2n+1}(\underline{1},\underline{2},\ldots,\underline{2}))$ or its dual;
 \item[{\rm F7}.] $\underline{2}(\underline{1}, T_{2n}(C_{3},\underline{2},\ldots,\underline{2}))$ or its dual;
 \item[{\rm F8}.] $ \underline{3}(\underline{1}, T_{2n+1}(\underline{1},\underline{1},\underline{2},\ldots,\underline{2}),\underline{1})$;
  \item[{\rm F9}.] $\underline{3}(\underline{1}, T_{2n}(C_{3},\underline{1},\underline{2},\ldots,\underline{2}),\underline{1})$;
  \item[{\rm F10}.] $\underline{3}(\underline{1},T_{2n-1}(C_{3},C_{3},\underline{2},\ldots,\underline{2}),\underline{1})$ when $n \geq 2$.
 \end{enumerate}
\end{thm}
\begin{proof}
Let  $T$ be a tournament with one of the forms $\text{F}4, \ldots, \text{F}11$. Clearly, $v(T) = 4n+2$ and $T$ admits a $\Delta$-decomposition of cardinality $2n+1$ or $2n+2$. Since $\Delta(4n+2) = 2n+2$ by Theorem~\ref{deltan}, then $\Delta(T)=\Delta(4n+2)$ or $\Delta(T)=\Delta(4n+2)-1$. It follows from Assertion~2 of Corollary~\ref{cor} that $\delta(T)=\delta(4n+2)$. Thus  $T \in \mathcal{T}_{\delta}^{4n+2}$.   
 
 Conversely, let $T$ be a tournament of $\mathcal{T}_{\delta}^{4n+2}$.
 By Assertion~2 of Corollary~\ref{cor}, $\Delta(T)=2n+1$ or $2n+2$. If $\Delta(T)=2n+2$, that is, $\Delta(T) = \Delta(4n+2)$ (see Theorem~\ref{deltan}), then $T$ is a tournament of the class $\mathcal{T}_{\Delta}^{4n+2}$. It follows from Proposition~\ref{thmDelta0}, that $T$ is $\underline{3}(\underline{1}, T_{2n}.\underline{2} ,\underline{1})$, which is a particular form of $\text{F}8$. 
 
 Now suppose $\Delta(T)=2n+1$. Let $D$ ba a $\delta$-decomposition of $T$. We have $|D| = 2n+1$. Recall that $|\text{sg}(D)| \leq 2$ (see Corollary~\ref{rcomod}).  
 
 First suppose $\text{sg}(D)=\varnothing$. Since $v(T) =4n+2$ and $|D|=\Delta(T)= 2n+1$, then $\cup D = V(T)$ and the elements of $D$ are of cardinality $2$. Moreover, it follows from Assertion 2 of Lemma~\ref{comod part} that there exists an element $I$ of $D$ such that
 \begin{equation}\label{eqi}
  \vert I\vert =2 \ \text{and} \text{ the elements of }D\setminus \{I\} \text{ are twins of }T. 
 \end{equation}
  If $I$ is a module of $T$, then the elements of $D$ are twins of $T$ (see (\ref{eqi})). Since $\cup D = V(T)$, it follows that $T$ has form $\text{F4}$. If $I$ is not a module of $T$, then since $\overline{I}$ is a module of $T$ and $\cup D = V(T)$, it follows from (\ref{eqi}) that $T$ has form $\text{F5}$.
  
  Second suppose $\text{sg}(D)\neq \varnothing$, that is, $|\text{sg}(D)|=1$ or $2$.  
  By Corollary~\ref{rcomod},
\begin{equation}\label{eqj}
 \text{the elements of }D_{\geq 2}\text{ are nontrivial modules of }T. 
 \end{equation}

 To begin, suppose  $\vert \text{sg}(D)\vert =1$. Since $v(T) = 4n+2$ and $|D| = \Delta(T) = 2n+1$, it follows from (\ref{eqj}) that there exists an element $I$ of $D_{\geq 2}$ such that 
 \begin{equation}\label{eqii}
  \vert I\vert = 2 \ \text{or} \ 3, \ \text{and} \text{ the elements of }D_{\geq 2}\setminus \{I\} \text{ are twins of }T.
  \end{equation}
  By (\ref{eqj}), (\ref{eqii}) and Observation~\ref{rr2}, $I$ is a twin of $T$ or $I$ is a module of $T$ such that $T[I]\simeq C_{3}$. If $I$ is a twin of $T$, then since $v(T) = 4n+2$, $|D| = 2n+1$ and $|\text{sg(D)}|=1$, it follows from (\ref{eqii}) that the elements of $D_{\geq 2}$ are twins of $T$ and that $|\cup D| = v(T)-1$. So $T$ has form $\text{F6}$. If $I$ is a module of $T$ such that $T[I]\simeq C_{3}$, then since $v(T) = 4n+2$, $|D| = 2n+1$ and $|\text{sg(D)}|=1$, it follows from (\ref{eqii}) that $\cup D = V(T)$ and the elements of $D \setminus \{I\}$ are twins of $T$. So $T$ has form $\text{F7}$.
  
   To finish, suppose  $\vert \text{sg}(D) \vert =2$. First suppose $\Delta(T) =3$. In this instance, $n =1$, i.e. $v(T)=6$. Thus, $T$ has form $\underline{3}(\underline{1}, T_4, \underline{1})$, where $T_4$ is a $4$-vertex tournament. Since the $4$-vertex tournaments are $\underline{4}$, $\underline{2}(C_3, \underline{1})$, $\underline{2}( \underline{1}, C_3)$, and $C_3(\underline{1}, \underline{1}, \underline{2})$ (see Fact~\ref{tournois4}), then $\Delta(T) =4$ if $T_4$ is $\underline{4}$, and $\Delta(T) =3$ otherwise.
   Since $\Delta(T)=3$, it follows that $T_4$ is $\underline{2}(C_3, \underline{1})$, $\underline{2}( \underline{1}, C_3)$, or $\underline{3}(\underline{1}, \underline{1}, \underline{2})$. If $T_4$ is $\underline{2}(C_3, \underline{1})$ or $\underline{2}(\underline{1}, C_3)$, then $T$ is $\underline{3}(\underline{1}, \underline{2}(C_3, \underline{1}), \underline{1})$ or its dual, and thus $T$ has form $\text{F}9$. If $T_4$ is $\underline{3}(\underline{1}, \underline{1}, \underline{2})$, then $T$ in $\underline{3}(\underline{1}, C_3(\underline{1}, \underline{1}, \underline{2}), \underline{1})$, and thus $T$ has form $\text{F}8$. 
   
   Second suppose $\Delta(T)\geq 4$. Since $v(T) = 4n+2$ and $|D| = 2n+1$, it follows from (\ref{eqj}) and from the second assertion of Observation~\ref{rr2} that there exist distinct $I,J \in D_{\geq 2}$ such that  
  \begin{equation}\label{eqsd22}
 |I| = 2 \ \text{or} \ 3, \ |J| = 2 \ \text{or} \ 3, \ \text{and} \ \text{ the elements of }D_{\geq 2}\setminus \{I, J\} \text{ are twins of }T.
   \end{equation}
  Since $I,J \in D_{\geq 2}$, then $I$ and $J$ are modules of $T$ (see (\ref{eqj})). Moreover, since $|I|, |J| \in \{2,3\}$, it follows from the first assertion of Observation~\ref{rr2} that by interchanging $I$ and $J$, one of the following three cases holds.
  \begin{itemize}
  	\item $I$ and $J$ are twins of $T$. In this instance, the elements of $D_{\geq 2}$ are twins of $T$ (see (\ref{eqj})). Since $|D| = 2n+1$ and $|\text{sg}(D)|=2$, we have $|D_{\geq 2}| = 2n-1$ and $|\cup D| = v(T)-2 = 4n$. Since the elements of $D_{\geq 2}$ are twins of $T$, it follows that $T$ has form~ $\text{F}8$.
    
    \item $I$ is a twin of $T$ and $J$ is a module of $T$ such that $T[J]\simeq  C_{3}$. Since $|D| = 2n+1$ and $|\text{sg}(D)|=2$, we have $|D_{\geq 2}| = 2n-1$ and $|\cup D| = v(T)-1$. Since the elements of $D_{\geq 2} \setminus \{J\}$ are twins of $T$, we obtain that $T$ has form~$\text{F}9$. 
   
    \item $I$ and $J$ are modules of $T$ such that $T[I]\simeq T[J]\simeq C_{3}$. In this instance, $v(T) \geq 8$, and since $v(T) =4n+2$, we have $n \geq 2$ and $v(T) \geq 10$. Since $|D| = 2n+1$ and $|\text{sg}(D)|=2$, we have $|D_{\geq 2}| = 2n-1$ and $\cup D = V(T)$. Since the elements of $D_{\geq 2} \setminus \{I,J\}$ are twins of $T$, we obtain that $T$ has form~$\text{F}10$ with $n \geq 2$. \qedhere
    \end{itemize}
\end{proof}

\begin{thm}\label{p4n+3}
Up to isomorphism,  the tournaments of the class $\mathcal{T}_{\delta}^{4n+3}$, where $n\geq 1$, are the tournaments with one of the following forms, where $T_{k}$ is  an arbitrary $k$-vertex tournament:
 \begin{enumerate}
 \item[{\rm F11}.] $T_{2n+2}(\underline{1},\underline{2},\ldots,\underline{2})$;
 \item[{\rm F12}.] $T_{2n+1}(C_{3},\underline{2},\ldots,\underline{2})$; 
\item[{\rm F13}.]  $C_{3}(\underline{1},\underline{1}, T_{2n}(C_{3},\underline{2},\ldots,\underline{2}))$; 
\item[{\rm F14}.] $C_{3}(\underline{1},\underline{1}, T_{2n+1}(\underline{1},\underline{2},\ldots,\underline{2}))$; 
\item[{\rm F15}.] $\underline{2}(\underline{1}, T_{2n+2}(\underline{1},\underline{1},\underline{2},\ldots,\underline{2}))$ or its dual;
 \item[{\rm F16}.] $\underline{2}(\underline{1},T_{2n+1}(C_{3},\underline{1},\underline{2},\ldots,\underline{2}))$ or its dual;
\item[{\rm F17}.]   $\underline{2}(\underline{1},T_{2n}(C_{3},C_{3},\underline{2},\ldots,\underline{2}))$ or its dual;
\item[{\rm F18}.]   $\underline{3}(\underline{1},T_{2n-1}(S_5,\underline{2},\ldots,\underline{2}),\underline{1})$, where $S_5=U_{5},V_{5}$ or $W_{5}$;
    \item[{\rm F19}.] $\underline{3}(\underline{1},T_{2n+2}(\underline{1},\underline{1},\underline{1},\underline{2},\ldots,\underline{2}),\underline{1})$; 
    \item[{\rm F20}.]$\underline{3}(\underline{1}, T_{2n}(C_{3},C_{3},\underline{1},\underline{2},\ldots,\underline{2}),\underline{1})$ when $n\geq 2$;  
 \item[{\rm F21}.] $\underline{3}(\underline{1},T_{2n+1}(\underline{1},\underline{1},C_{3},\underline{2},\ldots,\underline{2}),\underline{1})$;
     
 \item[{\rm F22}.] $\underline{3}(\underline{1},T_{2n-1}(C_{3},C_{3},C_{3},\underline{2},\ldots,\underline{2}) ,\underline{1})$ when $n\geq 2$.  
  \end{enumerate}
\end{thm}
\begin{proof}
Let  $T$ be a tournament with one of the forms $\text{F}11, \ldots, \text{F}22$. Clearly, $v(T) = 4n+3$ and $T$ admits a $\Delta$-decomposition of cardinality $2n+1$ or $2n+2$. Since $\Delta(4n+3) = 2n+2$ by Theorem~\ref{deltan}, then $\Delta(T)=\Delta(4n+2)$ or $\Delta(T)=\Delta(4n+2)-1$.
It follows from Assertion~2 of Corollary~\ref{cor} that $\delta(T) = \delta(4n+3)$, and hence $T \in \mathcal{T}_{\delta}^{4n+3}$.   

Conversely, let $T$ be a tournament of $\mathcal{T}_{\delta}^{4n+3}$. By Assertion~2 of Corollary~\ref{cor}, we have $\vert D\vert=\Delta(T)=2n+1$ or $2n+2$. To begin, suppose $\Delta(T)=2n+2$, that is, $\Delta(T) = \Delta(4n+3)$ (see Theorem~\ref{deltan}). In this instance, $T$ is a tournament of the class $\mathcal{T}_{\Delta}^{4n+3}$. It follows from Proposition~\ref{thmDelta1}, that $T$ is $\underline{2}(\underline{1},T_{2n+1}.\underline{2})$ or its dual,
 $\underline{3}(\underline{1},T_{2n+1}(\underline{1},\underline{2},\ldots,\underline{2}),\underline{1})$,
or $\underline{3}(\underline{1},T_{2n}(C_{3},\underline{2},\ldots,\underline{2}),\underline{1})$. In the first instance $T$ has a particular form of $\text{F}15$, in the second and third instances, $T$ has a particular form a $\text{F}19$.

Now suppose $\Delta(T)=2n+1$. Let $D$ ba a $\delta$-decomposition of $T$. We have $|D| = 2n+1$. Recall that $|\text{sg}(D)| \leq 2$ (see Corollary~\ref{rcomod}). 
  
First suppose $\text{sg}(D)=\varnothing$. Since $v( T) =4n+3$ and $|D| = 2n+1$, then by Assertion~2 of Lemma~\ref{comod part}, there exist distinct $I,J \in D$ such that 
 \begin{equation}\label{eq3i}
\vert I\vert =2, \ \vert J\vert =2 \text{ or } 3, \text{  and the elements of }D\setminus \{I,J\} \text{ are twins of }T.
  \end{equation}
 We distinguish the following cases.
 \begin{itemize}
 	\item Suppose $|J|=2$. By Assertion~2 of Lemma~\ref{comod part}, and by interchanging $I$ and $J$, we may assume that $I$ is a twin of $T$. By (\ref{eq3i}), the element of $D \setminus \{J\}$ are twins of $T$. Since $|\cup D|=v(T)-1 = 4n+2$, it follows that $T$ has form $\text{F}11$ if $J$ is also a twin of $T$. Moreover, if $J$ is not a module of $T$, then since $\overline{J}$ is a module of $T$ and $\text{sg}(D) = \varnothing$, we obtain that $T$ has form $\text{F}14$. 
 	     
 	\item Suppose $|J|=3$. Suppose for a contradiction that $J$ is not a module of $T$. Since $\overline{J}$ is a module of $T$, there exists a $4$-vertex tournament $Q_4$ such that $T$ is $Q_4(\underline{1}, \underline{1}, \underline{1}, T[\overline{J}])$. Since the $4$-vertex tournaments are decomposable, and since $J$ is not a module of $T$, it follows that there exists a proper subset $K$ of $J$ such that $K$ is a co-module of $T$. This contradicts the minimality of the co-module $J$ of $K$. Thus, $J$ is a module of $T$. On the other hand, $T[J] \simeq C_3$ by Assertion~1 of Observation~\ref{rr2}. Moreover, $\cup D = V(T)$. First, suppose that $I$ is a module of $T$. In this instance, the elements of $D \setminus \{J\}$ are twins of $T$ (see (\ref{eq3i})). Since $J$ is a module of $T$ such that $T[J] \simeq C_3$ and since $\cup D = V(D)$, it follows that $T$ has form $\text{F}12$. Second, suppose that $I$ is not a module of $T$. In this instance, $\overline{I}$ is a module of $T$. Moreover, the elements of $D \setminus \{I,J\}$ are twins of $T$ (see (\ref{eq3i})). Since $J$ is a module of $T$ such that $T[J] \simeq C_3$ and since $\cup D = V(D)$ and $\text{sg}(D) = \varnothing$, it follows that $T$ has form $\text{F}13$.         
 \end{itemize}
  
  Second suppose $\text{sg}(D)\neq \varnothing$, that is, $|\text{sg}(D)|=1$ or $2$.  
  By Corollary~\ref{rcomod},
\begin{equation}\label{eq33}
 \text{the elements of }D_{\geq 2}\text{ are nontrivial modules of }T. 
 \end{equation}

 To begin, suppose  $\vert \text{sg}(D)\vert =1$.
   Since $v(T) = 4n+3$ and $|D| = 2n+1$, it follows from (\ref{eq33}) and from the second assertion of Observation~\ref{rr2} that there exist distinct $I,J \in D_{\geq 2}$ such that  
  \begin{equation}\label{eqsd33}
 |I| = 2 \ \text{or} \ 3, \ |J| = 2 \ \text{or} \ 3, \ \text{and} \ \text{ the elements of }D_{\geq 2}\setminus \{I, J\} \text{ are twins of }T.
   \end{equation}
   Since $I,J \in D_{\geq 2}$, then $I$ and $J$ are modules of $T$ (see (\ref{eq33})). Moreover, since $|I|, |J| \in \{2,3\}$, it follows from the first assertion of Observation~\ref{rr2} that by interchanging $I$ and $J$, we may assume that $I$ and $J$ are twins of $T$, $I$ is  a twin of $T$  and $T[J]\simeq  C_{3}$, or  $T[I]\simeq T[J]\simeq C_{3}$.
   \begin{itemize}
    \item Suppose that $I$ and $J$ are twins of $T$. In this instance, the elements of $D_{\geq 2}$ are twins of $T$ (see (\ref{eq33})). Since $|D| = 2n+1$ and $|\text{sg}(D)|=1$, we have $|D_{\geq 2}| = 2n$ and $|\cup D| = v(T)-2 = 4n+1$. Since the elements of $D_{\geq 2}$ are twins of $T$, it follows that $T$ has form~$\text{F}15$. 
    
    \item Suppose that $I$ is a twin of $T$ and $J$ is a module of $T$ such that $T[J]\simeq  C_{3}$. Since $|D| = 2n+1$ and $|\text{sg}(D)|=1$, we have $|D_{\geq 2}| = 2n$ and $|\cup D| = v(T)-1 =4n+2$. Since the elements of $D_{\geq 2} \setminus \{J\}$ are twins of $T$, we obtain that $T$ has form~ $\text{F}16$. 
    \item Suppose that $I$ and $J$ are modules of $T$ such that $T[I]\simeq T[J]\simeq C_{3}$. Since $|D| = 2n+1$ and $|\text{sg}(D)|=1$, we have $|D_{\geq 2}| = 2n$, and $\cup D = V(T)$. Since the elements of $D_{\geq 2} \setminus \{I,J\}$ are twins of $T$, we obtain that $T$ has form~$\text{F}17$.
    \end{itemize}
     
 To finish, suppose $\vert \text{sg}(D)\vert =2$. 
 First suppose $\Delta(T)=3$ or $4$. In this instance, $n=1$, i.e., $v(T)=7$. Thus $T = \underline{3}(\underline{1}, X_5, \underline{1})$, where $X_5$ is a $5$-vertex tournament. Recall that $|D|= \Delta(T) = 2n+1 =3$. If $X_5$ is $\underline{5}$, $\underline{3}(\underline{1}, \underline{1}, C_3)$, or $\underline{3}(C_3, \underline{1}, \underline{1})$, then $\Delta(T)=4$ (e.g. see Proposition~\ref{thmDelta1}). Since $\Delta(T)=3$, it follows from Fact~\ref{tournois5} that $X_5$ is 
  $\underline{3}(\underline{1}, C_3,\underline{1})$, $\underline{2}(\underline{1}, C_4)$, $\underline{2}(C_4, \underline{1})$, $C_3(\underline{1}, \underline{2}, \underline{2})$, $C_3(\underline{1}, \underline{1}, C_3)$, $C_3(\underline{1}, \underline{1}, \underline{3})$, $U_5$, $V_5$, and $W_5$ (see Fact~\ref{tournois5}). If $X_5$ is $U_5$, $V_5$, or $W_5$, then $T$ has form $\text{F}18$. If $X_5$ is $\underline{2}(\underline{1}, C_4)$, $\underline{2}(C_4, \underline{1})$, $C_3(\underline{1}, \underline{2}, \underline{2})$, or $C_3(\underline{1}, \underline{1}, \underline{3})$, then $T$ has form $\text{F}19$. If $X_5$ is $\underline{3}(\underline{1}, C_3,\underline{1})$, or $C_3(\underline{1}, \underline{1}, C_3)$, then $T$ has form $\text{F}21$.
       
 Second suppose $\Delta(T)\geq 5$. In this instance, $n \geq 2$ and hence $v(T) \geq 11$. To begin, suppose there exists $M \in D_{\geq 2}$ such that $|M| \geq 4$. By Assertion~2 of Observation~\ref{rr2}, we have $|M| \geq 5$. Since $|D_{\geq 2}| = 2n-1$ and $|\cup (D_{\geq 2} \setminus \{M\})| \leq 4n-4$, we have $|\cup (D_{\geq 2} \setminus \{M\})| = 4n-4$ and $|M|=5$. Hence $\cup D = V(T)$. Moreover, the element of $D_{\geq 2} \setminus \{M\}$ are twins of $T$ (see (\ref{eq33})). On the other hand, by Assertion~3 of Observation~\ref{rr2}, $T[M]$ is isomorphic to $U_5$, $V_5$, or $W_5$. Thus, $T$ has form $\text{F}18$. 
 
 Now, suppose that every element of $D_{\geq 2}$ has cardinality at most $3$. Since 
 \begin{equation} \label{donne}
 v(T)=4n+3, \  |\text{sg}(D)|=2, \ \text{and} \ \vert D_{\geq 2}\vert = 2n-1,
 \end{equation}
 $D_{\geq 2}$ admits at most three elements of cardinality $3$. It follows from (\ref{eq33}) that there exist pairwise distinct $I,J,K \in D_{\geq 2}$ such that  
  \begin{equation}\label{eqsd4}
 |I|, |J|, |K| \in \{2,3\}, \ \text{and} \ \text{ the elements of }D_{\geq 2}\setminus \{I, J, K\} \text{ are twins of }T.
   \end{equation}
    Since $I,J,K \in D_{\geq 2}$, then $I,J$ and $K$ are modules of $T$ (see (\ref{eq33})). Moreover, if for $N \in \{I,J,K\}$, we have $|N|=3$, then $T[N] \simeq C_3$ by Assertion~1 of Observation~\ref{rr2}. Thus, 
    by interchanging $I$, $J$ and $K$, we only have to distinguish the following four cases.
    \begin{itemize}
    	\item Suppose that $I$, $J$ and $K$ are twins of $T$. In this instance, since the elements of $D_{\geq 2}$ are twins of $T$ (see (\ref{eqsd4})), it follows from (\ref{donne}) that $T$ has form~ $\text{F}19$.
    	
    	\item Suppose that $I$ and $J$ are twins of $T$, and that $K$ is a module of $T$ such that $T[K] \simeq C_3$. Since the elements of $D_{\geq 2} \setminus\{K\}$ are twins of $T$ (see (\ref{eqsd4})), it follows from (\ref{donne}) that $T$ has form $\text{F}21$.
    	
    \item Suppose that $I$ is a twin of $T$, and that $J$ and $K$ are modules of $T$ such that $T[J] \simeq T[K] \simeq C_3$. Since the elements of $D_{\geq 2} \setminus \{J,K\}$ are twins of $T$ (see (\ref{eqsd4})), it follows from (\ref{donne}) that $T$ has form $\text{F}20$ with $n \geq 2$.
    	
     \item Suppose that $I$, $J$ and $K$ are modules of $T$ such that $T[I] \simeq T[J] \simeq T[K] \simeq C_3$. Since the elements of $D_{\geq 2} \setminus \{I,J,K\}$ are twins of $T$ (see (\ref{eqsd4})), it follows from (\ref{donne}) that $T$ has form $\text{F}22$ with $n \geq 2$. 
    	 \qedhere
    \end{itemize} 
  \end{proof}   
 

{}
\end{document}